\documentclass[12pt]{amsart}
\usepackage{amsmath,amssymb,amsfonts,amsthm,amscd,amstext,amsxtra,amsopn,array,url,verbatim,mathrsfs,enumerate,anysize,enumitem}
\usepackage{graphicx}
\usepackage{amsmath,amssymb,amsfonts,amsthm,amssymb,amscd,url,amstext,amsxtra,amsopn
,txfonts}
\usepackage{verbatim}
\marginsize{2.20cm}{2.20cm}{2.20cm}{2.20cm}
\usepackage{times}
\usepackage{amsmath}
\usepackage{amssymb}
\usepackage{amsfonts}
\usepackage{xcolor}
\usepackage{hyperref}

\newtheorem{theorem}{Theorem}[section]
\newtheorem{definition}[theorem]{Definition}
\newtheorem{proposition}[theorem]{Proposition}
\newtheorem{lemma}[theorem]{Lemma}

\newtheorem{corollary}[theorem]{Corollary}

\newtheorem{remark}[theorem]{Remark}
\newtheorem{remarks}[theorem]{Remarks}
\numberwithin{equation}{section}

\DeclareMathOperator{\Gal}{Gal}

\DeclareMathOperator{\modd}{mod}
\DeclareMathOperator{\GL}{GL}

\usepackage{mathtools}

\begin{document}

\begin{center}

\title{On the moments of torsion points modulo primes and their applications}
\author{Amir Akbary}
\address{Department of Mathematics and Computer Science, University of Lethbridge, Lethbridge, Alberta T1K 3M4, Canada}
\email{amir.akbary@uleth.ca}

\author{Peng-Jie Wong}
\address{Department of Mathematics and Computer Science, University of Lethbridge, Lethbridge, Alberta T1K 3M4, Canada}
\email{pengjie.wong@uleth.ca} 

\thanks{Research of the first author is partially supported by NSERC. Research of the second author is partially supported by a PIMS postdoctoral fellowship.}

\date{\today}

\begin{abstract}
Let $\mathbb{A}[n]$ be the group of $n$-torsion points of a commutative algebraic group $\mathbb{A}$ defined over a number field $F$. For a prime ideal $\mathfrak{p}$, 
we let $N_{\mathfrak{p}}(\mathbb{A}[n])$
 be the number of  $\mathbb{F}_\mathfrak{p}$-solutions of the system of polynomial equations defining $\mathbb{A}[n]$ when reduced modulo $\mathfrak{p}$. Here, $\mathbb{F}_{\mathfrak{p}}$ is the residue field at $\mathfrak{p}$. Let $\pi_F(x)$ denote the number of primes $\mathfrak{p}$ of $F$ whose norm $N(\mathfrak{p})$ do not exceed $x$. We then, for algebraic groups of dimension one,  compute the $k$-th moment limit 
$$M_k(\mathbb{A}/F, n)=\lim_{x\rightarrow \infty} \frac{1}{\pi_F(x)} \sum_{N(\mathfrak{p}) \leq x} N_{\mathfrak{p}}^k(\mathbb{A}[n])$$
by appealing to the prime number theorem for arithmetic progressions and more generally the Chebotarev density theorem. We further interpret this limit as the number of orbits of the action of the absolute Galois group of $F$ 
on $k$ copies of $\mathbb{A}[n]$ by an application of Burnside's Lemma.
These concrete examples suggest a possible approach for determining the number of orbits of a group acting on $k$ copies of a set.  We also show that for an algebraic set $Y$ of dimension zero, the corresponding arithmetic function $N_\mathfrak{p}(Y)$,
defined on primes  $\mathfrak{p}$ of $F$, has an asymptotic limiting distribution.
\end{abstract}
\maketitle

\end{center}

\section{Introduction}\label{S1}

Let $\mathbb{A}$ be a commutative algebraic group defined over a number field $F$. We let $\mathbb{A}[n]$  be the group of $n$-torsion points of $\mathbb{A}$ and $F(\mathbb{A}[n])$  be the field generated by adding the coordinates of $\mathbb{A}[n]$ to $F$. For a prime  $\mathfrak{p}$ of $F$ that is unramified in $F(\mathbb{A}[n])/F$, let $\mathbb{F}_\mathfrak{p}$ denote the residue field at $\mathfrak{p}$, and let 
$N_{\mathfrak{p}}(\mathbb{A}[n])$
 be the number of  $\mathbb{F}_\mathfrak{p}$-solutions of the system of polynomial equations defining $\mathbb{A}[n]$ when reduced modulo $\mathfrak{p}$.  If $\mathfrak{p}$ ramifies, we set $N_{\mathfrak{p}}(\mathbb{A}[n])=0$. In order to investigate the average size of $N_{\mathfrak{p}}(\mathbb{A}[n])$,  we set
\begin{equation}
\label{limit}
M(\mathbb{A}/F, n)=\lim_{x\rightarrow \infty} \frac{1}{\pi_F(x)} \sum_{N(\mathfrak{p}) \leq x} N_{\mathfrak{p}}(\mathbb{A}[n]),
\end{equation}
where $\pi_F(x)$ denotes the number of primes $\mathfrak{p}$ of $F$ whose norm $N(\mathfrak{p})$ do not exceed $x$. 

In \cite{CK12}, Chen and Kuan investigated the average size of the arithmetic function $N_{\mathfrak{p}}(\mathbb{A}[n])$ by determining $M(\mathbb{A}/F, n)$ as the number of orbits of the group ${\rm Gal}(F(\mathbb{A}[n])/F)$ acting on the $n$-torsion points $\mathbb{A}[n]$ (see \cite[Theorem 1.2]{CK12}). Moreover, they showed that for commutative algebraic groups of dimension one other than $\mathbb{G}_a$, the value of $M(\mathbb{A}/F, n)$ is given by a divisor function.
More precisely, it is known that a commutative algebraic group of dimension one over $F$ is either the additive group $\mathbb{G}_a$, the multiplicative group $\mathbb{G}_m$, an algebraic torus of dimension one, or an elliptic curve. For $\mathbb{G}_a$ we have $M(\mathbb{G}_a/F, n)= 1$. For other cases, the following assertions are proved in \cite[Corollary 1.3, Theorem 1.4, Corollary 1.5, and Theorem 1.6]{CK12}. Here, $\zeta_n$ denotes a primitive $n$-th root of unity and $d(n)$ is the number of positive divisors of $n$.

\begin{theorem}[Chen-Kuan]  \label{Chen-Kuan} (i) Assume that $F\cap \mathbb{Q}(\zeta_n)=\mathbb{Q}$. Then $M(\mathbb{G}_m/F, n)=d(n).$

(ii) 
Let $\mathbb{T}$ denote a one-dimensional torus  over $\mathbb{Q}$. Then there is a positive constant $C:= C(\mathbb{T})$, depending only on $\mathbb{T}$, such that for $n$ with $(n, C)=1$, one has $M(\mathbb{T}/\mathbb{Q}, n)=d(n).$ 

(iii) Assume that $E$ is a non-CM elliptic curve defined over $F$.  Then there is a positive constant $C:= C(E, F)$, depending only on $E$ and $F$,  such that for $n$ with $(n, C)=1$, one has  $M(E/F, n)=d(n)$.  

(iv) Assume that $E$ is an elliptic curve defined over $F$ which has complex multiplication by an order in an imaginary quadratic field $K$.  Assume $FK\cap \mathbb{Q}(\zeta_n)=\mathbb{Q}$. Then there is a positive constant $C:= C(E, F)$, depending only on $E$ and $F$,  such that for $n$ with $(n, 2C)=1$, one has  
$$M(E/F, n)=\begin{cases} d_K(n)&{\rm if}~ K\subseteq F,\\
\frac{1}{2} (d_K(n)+d(n)) &{\rm if}~K\not\subseteq F. 
\end{cases}
$$
Here $d_K(n)$ denotes the number of ideal divisors of the ideal $n\mathcal{O}_K$ in $\mathcal{O}_K$, the ring of integers of $K$. The conditions $FK\cap \mathbb{Q}(\zeta_n)=\mathbb{Q}$  and $(n, 2)=1$ only apply to the case that $K\not\subseteq F$.

\end{theorem}

\begin{remarks} (i) In \cite{CK12} the function $N_{\mathfrak{p}}(\mathbb{A}[n])$ is defined, for a prime $\mathfrak{p}$ of good reduction of $\mathbb{A}$, as the number of $n$-torsion points in the group of $\mathbb{F}_\mathfrak{p}$-rational  points of the reduction modulo $\mathfrak{p}$ of $\mathbb{A}$. Our definition of $N_{\mathfrak{p}}(\mathbb{A}[n])$ may differ from that definition only at finitely many prime ideals $\mathfrak{p}$, and thus it will not affect the assertions of Theorem \ref{Chen-Kuan}.

(ii) Parts (iii) and (iv) of Theorem \ref{Chen-Kuan} are also stated and proved in \cite[Corollaries 1, 3, and 4]{H14}.

(iii) The conditions $FK \cap \mathbb{Q}(\zeta_n)=\mathbb{Q}$ and $(n, 2)=1$ in part (iv) of Theorem \ref{Chen-Kuan} is not clearly stated in  \cite[Theorem 1.6]{CK12}; however, these conditions are used in the proof of Theorem 1.6 in \cite{CK12}.

(iv) In \cite[Theorem 1.4]{CK12} it is also proved that the constant $C$ in part (ii) of Theorem \ref{Chen-Kuan} can be taken as $1$ if $m>0$ and as $D_m$ if $m<0$, where $m$ is the square-free integer in the equation $x^2-my^2=1$ defining $\mathbb{T}$, and $D_m$ is the discriminant of the quadratic field $\mathbb{Q}(\sqrt{m})$. Also, it is shown, for $F=\mathbb{Q}$,  that in part (iv) of Theorem \ref{Chen-Kuan} the constant $C$ can be taken as $6\Delta_E$, where $\Delta_E$ is the discriminant of $E$ (see \cite[Theorem 1.6]{CK12}). In addition, the extensions of Theorem \ref{Chen-Kuan} to the case of function fields are given in \cite{CK15}. 

\end{remarks}

The proof of the first three parts of Theorem \ref{Chen-Kuan} can be unified and simplified considerably if one interprets the limit \eqref{limit} as the number of the orbits of  ${\rm GL}_m (\mathbb{Z}/n\mathbb{Z})$, the group of invertible $m\times m$ matrices with entries in $\mathbb{Z}/n\mathbb{Z}$, acting on the product of $m$ copies of  $\mathbb{Z}/n\mathbb{Z}$, when $m=1$ or $2$. In this direction, the following can be considered as a generalization of the underlying result in parts (i), (ii), and (iii) of Theorem \ref{Chen-Kuan}.

\begin{theorem}
\label{first generalization}
Let $L$ be a number field of class number 1.  Then the number of orbits of ${\rm GL}_m (\mathcal{O}_L/n\mathcal{O}_L)$ acting on $\left(\mathcal{O}_L/n\mathcal{O}_L\right)^m$ is $d_L(n)$, where $d_L(\cdot)$ is the number field analogue of the divisor function.
\end{theorem}

In another direction, as a consequence of the results of this paper, we give a generalization of Theorem \ref{Chen-Kuan} by considering the $k$-th moment limit
$$M_k(\mathbb{A}/F, n)=\lim_{x\rightarrow \infty} \frac{1}{\pi_F(x)} \sum_{N(\mathfrak{p}) \leq x} N_{\mathfrak{p}}^k(\mathbb{A}[n]).$$

Note that, for every $k\geq 1$,  $M_k(\mathbb{G}_a/F, n)=1$. In order to state our result for other algebraic groups of dimension one, we need to introduce the following notation. For $k\in \mathbb{Z}^{\geq 0}$ and $n\in \mathbb{N}$, let 

$$M_k(n):= \sum_{\substack{d,e\\de\mid n}}\frac{d^k \mu(e)}{\varphi(de)},$$
where $\mu$ is the M\"{o}bius function, and $\varphi$ is the Euler function. Observe that for $a, b\in \mathbb{N}$ and integer $k\geq 0$, by letting $$P_k(a, b)=\frac{a^k-b^k}{a-b},$$
we have $$M_k(n)=\prod_{\ell^s\| n} \left( \sum_{e=1}^s P_k(\ell^e, \ell^{e-1})+1 \right).$$
Note that $M_0(n)=1$ and $M_1(n)=d(n)$. Thus,  $M_k(n)$ can be considered as a generalization of the divisor function.

We have the following generalization of Theorem \ref{Chen-Kuan}.

\begin{theorem}  
\label{second generalization}
(i) Assume that $F\cap \mathbb{Q}(\zeta_n)=\mathbb{Q}$. Then $M_k(\mathbb{G}_m/F, n)=M_k(n).$

(ii) 
Let $\mathbb{T}$ be a one-dimensional torus defined over $\mathbb{Q}$. Then there is a positive constant $C:= C(\mathbb{T})$, depending only on $\mathbb{T}$, such that for $n$ with $(n, C)=1$, we have $M_k(\mathbb{T}/\mathbb{Q}, n)=M_k(n).$ 

(iii) Assume that $E$ is a non-CM elliptic curve defined over $F$.  Then there is a positive constant $C:= C(E, F)$, depending only on $E$ and $F$,  such that for square-free $n$ with $(n, C)=1$, we have  

$$M_k(E/F, n)=\prod_{\ell \mid n}  \frac{\ell^{2k-1}+ \ell^{k-1}( \ell^3-2\ell-1) +\ell^3-2\ell^2-\ell+3}{(\ell-1)^2(\ell+1)}.$$

(iv) Assume that $E$ is an elliptic curve defined over $\mathbb{Q}$ that has complex multiplication by $\mathcal{O}_K$.
Then there is a positive constant $C:= C(E)$, depending only on $E$,  such that for prime $\ell$ with $(\ell, 2C)=1$, we have

$$M_k(E/\mathbb{Q}, \ell)
=  \frac{\ell^{2k}+ (d_K(\ell)-1)(\ell^{k+1}+\ell^k)+2\ell^2-(d_K(\ell)-1)\ell-(d_K(\ell)+2)}{2(\ell^2-1)}.
$$
\end{theorem}

\begin{remark}
For $k\geq 3$, the $\ell$- factor in the product expression for $M_k(E/F, n)$ in part (iii) of the above theorem is a polynomial function of degree $2k-4$ of $\ell$ with integral coefficients. For $k=1$ (resp. $k=2$), the $\ell$-factor is $2$ (resp. $\ell+3$). The expression in part (iv) is a polynomial function of degree $2k-2$ of $\ell$ with half-integral coefficients.
\end{remark}

Theorem \ref{second generalization}, similarly to Theorem \ref{Chen-Kuan},  is intimately related to a group theory result. In order to describe the connection, we introduce a more general setup.

Let $\overline{F}$ denote the algebraic closure of a number field $F$. Let $Y$ be an algebraic set (affine or projective), given as the set of $\overline{F}$-solutions of a finite family of polynomial equations 
$E_Y$ 
defined over the ring  of integers $\mathcal{O}_F$ of $F$.  (If $Y$ is projective, ``polynomial equations" means ``homogeneous polynomial equations" and ``$\mathbb{F}_{\mathfrak{p}}$-solutions" means ``projective $\mathbb{F}_{\mathfrak{p}}$-solutions".) For an unramified  prime ideal $\mathfrak{p}$ in the extension $F(Y)/F$, we let 
$$N_{\mathfrak{p}}(Y):=\#\{ {\rm solutions~ of~ }E_Y (\modd \mathfrak{p})~{\rm in}~ \mathbb{F}_\mathfrak{p}\}.$$ 
If $Y$ is the set of $\overline{F}$-solutions of a single polynomial $f$, we also denote $N_{\mathfrak{p}}(Y)$ by $N_{\mathfrak{p}}(f)$. 

\begin{remark}
Theorem 1.2 (c) of \cite{Se12}  provides a generalization of Theorem \ref{Chen-Kuan} and another interpretation for the limit \eqref{limit} for the case $F=\mathbb{Q}$. For an algebraic set $Y$ defined over $\mathbb{Z}$, let $N_p(Y)$ be as defined above.
 Then if the dimension ${\rm dim} Y(\mathbb{C}) \leq d_0$, one has 
$$\lim_{x\rightarrow \infty} \frac{1}{\pi(x^{d_0+1})} \sum_{p\leq x} N_p(Y)=r_0(Y),$$
where $r_0(Y)$ is the number of $\mathbb{Q}$-irreducible components of dimension $d_0$ of $Y$ over 
$\mathbb{Q}$. Here, $\pi(x):= \pi_\mathbb{Q}(x)$.
Note that for $d_0=0$, the above limit is analogous to the one evaluated in Theorem \ref{Chen-Kuan}. For example, for the algebraic set $Y$ defined by $x^n-1= \prod_{d\mid n} \Phi_d(x)$, where $\Phi_d(x)$ is the $d$-th cyclotomic polynomial, we have $r_0(Y)=d(n)$.  
\end{remark}

We now assume that $Y$ has dimension zero (so it is finite) and let $M_k(G, Y)$ be the  the number of orbits of $G={\rm Gal}(F(Y)/F)$ acting on $k$ copies of $Y$. Since there are only finitely many prime ideals that ramify in $F(Y)/F$, for a ramified prime ideal $\mathfrak{p}$ we define  $N_{\mathfrak{p}}(Y)=0$ for convenience. The following main result represents $M_k(G, Y)$ as an asymptotic average of the values $N_{\mathfrak{p}}^{k}(Y)$ as $\mathfrak{p}$ varies over the set of primes of $F$.

\begin{theorem}
\label{main}
Let $Y$ be an algebraic set of dimension zero defined over $F$,  $G={\rm Gal}(F(Y)/F)$, and $M_k(G, Y)$ as defined above. Then, for $k\in\Bbb{N}$, we have
$$
\lim_{x\rightarrow \infty}\frac{1}{\pi_{F}(x)}\sum_{N(\mathfrak{p})\le x}N_{\mathfrak{p}}^k(Y)= M_k(G, Y).
$$

\end{theorem}

The above theorem can be considered as a generalization of a classical result due to Frobenius and Kronecker (see \cite[p. 436]{Se03}).

\begin{theorem}[Frobenius-Kronecker]
\label{Frobenius-Kronecker}
For an irreducible polynomial  $f\in \Bbb{Z}[x]$,  we have
$$
\lim_{x\rightarrow \infty} \frac{1}{\pi(x)} \sum_{p\le x }N_p(f)=1.
$$
\end{theorem}

Indeed, let $F=\mathbb{Q}$, $Y=$ the set of roots of $f$ in $\overline{\mathbb{Q}}$, $k=1$, and $G={\rm Gal}(F(Y)/F)$ in Theorem \ref{main}. Then, observing that the action of the Galois group
 on the set of roots of $f$ is transitive, 
we obtain Theorem \ref{Frobenius-Kronecker} as a corollary of Theorem \ref{main}. 
Note that although the action of $G$ on $Y$ in Theorem \ref{Frobenius-Kronecker} is transitive,   the action on $k\geq 2$ copies of $Y$ is not transitive if $|Y|>1$.
 Thus, determining $M_k(G, Y)$ appears to be a non-trivial  problem for $k\geq 2$, even when $Y$ is defined by an irreducible polynomial.

As a direct consequence of Theorem \ref{main}, we  establish the existence of an asymptotic distribution function for the arithmetic function $N_\mathfrak{p}(Y)$.

\begin{corollary}
\label{distribution}
Let $Y$ be an algebraic set of dimension zero defined over $F$. Then the arithmetic function $N_\mathfrak{p}(Y)$ possesses an asymptotic distribution function. In other words, the sequence
$$H_{n}(z)=\frac{\#\{\mathfrak{p};~N(\mathfrak{p})\leq n~ {\rm and}~ N_\mathfrak{p}(Y) \leq z \}}{\pi_F(n)}$$ 
converges weakly to a distribution function $H$, as $n\rightarrow \infty$ (i.e., there is a distribution function $H$ where $H_n(z)$ converges pointwise to $H(z)$ at any continuity point $z$ of $H$). Moreover, for complex $t$-values $|t|<1$,
$$\varphi_H(t)= \lim_{n\rightarrow \infty} \frac{1}{\pi_F(n)}{\displaystyle{\sum_{N(\mathfrak{p}) \leq n}} e^{i tN_\mathfrak{p}(Y)}      }=\sum_{k=0}^{\infty} M_k(G, Y)\frac{(it)^k}{k!},$$ 
where $G={\rm Gal}(F(Y)/Y)$, and $\varphi_H(t)$ is the characteristic function of $H$.
\end{corollary}

We next describe that how Theorem \ref{main} can be exploited to answer some pure group-theoretic questions.
A fundamental question regarding the action of a group $G$ on a set $X$ is to determine the number of orbits in $X$ under the action of $G$. Moreover, if the number of orbits in $X$ under the action of $G$ is known, one may further ask whether there exists a formula for $M_k(G, X)$, the number of orbits in $k$ copies of $X$ under the action of $G$. Indeed, both are deep questions. 
Here, we show that how Theorem \ref{main} can be employed in computing $M_k(G, X)$.
The following definition describes our setup.

\begin{definition}
\label{arithmetic}
An action of a finite group $G$ on a finite set $X$ is called ``{arithmetically realizable over a number field $F$}", if there is a set $Y$ of solutions of a finite family of equations defined over $\mathcal{O}_F$, a bijection $\psi$ from $X$ to $Y$, and a group isomorphism $\phi$ from $G$ to ${\rm Gal}(F(Y)/F)$ such that $\psi (gx)=\phi(g)\psi(x)$.
\end{definition}

Inspiring by this definition, we can rewrite Theorem \ref{main} as the following.

\medskip\par

\noindent{\bf Theorem \ref{main} (Second Version)} {\it Suppose that the finite group $G$ has an action on a finite set $X$ that is arithmetically realizable over $F$. Let $Y$ be as given  in Definition \ref{arithmetic}. Then, for any $k\in\Bbb{N}$, we have
$$
M_k(G, X)=\lim_{x\rightarrow \infty}\frac{1}{\pi_{F}(x)}\sum_{N(\mathfrak{p})\le x}N_{\mathfrak{p}}^k(Y).$$

}

\medskip\par

This formulation of Theorem \ref{main} provides a line of approach in computing $M_k(G, X)$ for an arithmetically realizable action.
Of course, more generally one can consider the problem of computing $M_k(G, X)$ for an action of a group $G$ on a set $X$. In this generality, the problem appears to be difficult, and  we refer the reader to Cameron's survey \cite{Ca00} for results regarding the computation of $M_k(G, X)$ when the action of a permutation group $G$ (finite or not) on a set X is oligomorphic (i.e., $G$ has only finitely many orbits in $X^k$ for all $k$).

Our purpose here is to demonstrate by some examples that for arithmetically realizable actions a number-theoretic approach via Theorem \ref{main}  and the Chebotarev density theorem might help one to compute $M_k(G, X)$. For instance, as a consequence of Propositions \ref{main-prop} and \ref{prop-two}, we have the following explicit values for $M_k(G, X)$. (In all cases below, the actions are considered multiplicatively and in (ii) also componentwise.).

\begin{theorem}
\label{second-theorem}
(i) If $G=(\Bbb{Z}/n\Bbb{Z})^\times$ and $X=\Bbb{Z}/n\Bbb{Z}$, we have $M_k(G, X)=M_k(n).$

(ii) 
Let $$G=\left\{ \left(\begin{array}{cc} 1&0\\b&d \end{array} \right);~b\in \Bbb{Z}/n\Bbb{Z}~{\textrm and}~d\in (\Bbb{Z}/n\Bbb{Z})^\times \right\}\simeq  (\Bbb{Z}/n\Bbb{Z})^\times  \ltimes \Bbb{Z}/n\Bbb{Z}.$$ 
If $X=\left( \{1\} \times \Bbb{Z}/n\Bbb{Z} \right)  \times  \left(\{0\} \times \Bbb{Z}/n\Bbb{Z}\right)$, then
$M_k(G, X)=M_{2k-1}(n) .$


(iii) For  prime $\ell$, if $G=\GL_2(\Bbb{Z}/\ell\Bbb{Z})$ and $X=\Bbb{Z}/\ell\Bbb{Z}\times\Bbb{Z}/\ell\Bbb{Z}$, then
$$M_k(G, X)= \frac{\ell^4-2\ell^3-\ell^2+3\ell}{(\ell^2-\ell)(\ell^2-1)}+\ell^k \frac{\ell^3-2\ell-1}{(\ell^2-\ell)(\ell^2-1)}+\ell^{2k}\frac{1}{(\ell^2-\ell)(\ell^2-1)}.$$

\end{theorem}

The proof of Theorem \ref{second-theorem} relies on explicit computations of the moment limit in Theorem \ref{main} for certain algebraic sets $Y$ via the prime number theorem in arithmetic progressions and more generally by the  Chebotarev density theorem. We summarize these concrete evaluations in Propositions \ref{main-prop} and \ref{prop-two}. 
For $n\in \mathbb{N}$ and integer $a\in \mathbb{Z}$, let $$f_{n, a}(x):= x^n-a.$$
%
We have the following.

\begin{proposition}
\label{main-prop}
Let $n$ be a natural number. 
Let $a$ be a square-free positive integer if $n$ is odd, and let $a$ be a square-free positive integer such that $a\nmid n$ if $n$ is even.
Then the following estimates hold:

(i) For $k\in \mathbb{Z}^{\geq 0}$, $n\in \mathbb{N}$, we have $$
\lim_{x\rightarrow \infty}\frac{1}{\pi(x)}\sum_{p\le x}N_{p}^k(f_{n,1})
=M_k(n).
$$

(ii) For $k\in \mathbb{N}$, $n \in \mathbb{N}$,  we have
$$
\lim_{x\rightarrow \infty}\frac{1}{\pi(x)}\sum_{p\le x}N_{p}^k(f_{n,a})
=M_{k-1}(n).
$$

(iii) For any $k_1\in\mathbb{N}$, $k_2\in\mathbb{Z}^{\geq 0}$,  we have
$$
\lim_{x\rightarrow \infty}\frac{1}{\pi(x)}\sum_{p\le x}N_{p}^{k_1}(f_{n,a})N_{p}^{k_2}(f_{n,1})
=M_{k_1+k_2-1}(n).
$$


\end{proposition}

We next let $E$ be an elliptic curve defined over $\mathbb{Q}$. For prime $\ell$ let 
$E[\ell]$ denote the group of $\ell$-torsion points of $E$. The following assertions hold.

\begin{proposition}
\label{prop-two}
(i) Assume that ${\rm Gal}(\mathbb{Q}(E[\ell])/\mathbb{Q})\simeq {\rm GL}_2(\mathbb{Z}/\ell \mathbb{Z})$. Then
$$
\lim_{x\rightarrow \infty}\frac{1}{\pi(x)}\sum_{p\le x}N_{p}^k(E[\ell])
=\frac{\ell^4-2\ell^3-\ell^2+3\ell}{(\ell^2-\ell)(\ell^2-1)}+\ell^k \frac{\ell^3-2\ell-1}{(\ell^2-\ell)(\ell^2-1)}+\ell^{2k}\frac{1}{(\ell^2-\ell)(\ell^2-1)} .
$$

(ii)  Let $E$ have complex multiplication by $\mathcal{O}_K$, the ring of integers of an imaginary quadratic field $K$. For a fixed odd prime $\ell$, assume that ${\rm Gal}({K}(E[\ell])/{K})\simeq {\rm GL}_1(\mathcal{O}_K/\ell \mathcal{O}_K)$. Then
$$
\lim_{x\rightarrow \infty}\frac{1}{\pi(x)}\sum_{p\le x}N_{p}^k(E[\ell])
=\frac{2\ell^2-(d_K(\ell)-1)\ell-(d_K(\ell)+2)}{2(\ell^2-1)}+\ell^k \frac{d_K(\ell)-1}{2(\ell-1)}+\ell^{2k}\frac{1}{2(\ell^2-1)},
$$
where $d_K(\ell)$ is the number field analogue of the divisor function. More precisely, $d_K(\ell)=4, 3, 2$ if $\ell$ splits, ramifies, or remains inert in $K$, respectively.

\end{proposition}

In the rest of the paper we prove our results. The structure of the paper is as follows. In Section \ref{Section 2} we give a proof of Theorem \ref{first generalization}. Section \ref{S3} provides a proof of our general result, Theorem \ref{main}, and Corollary \ref{distribution}. In Section \ref{S2}, we compute some concrete examples of the $k$-th moment in Theorem \ref{main} by appealing to the prime number theorem in arithmetic progressions and the Chebotarev density theorem (Propositions \ref{main-prop} and \ref{prop-two}). Combining the results proved in Sections \ref{S3} and \ref{S2},  in Section \ref{Section 5}, by proving Theorem \ref{second-theorem}, we compute the number of orbits of certain finite groups acting on product of $k$ copies of certain finite sets. Finally, in Section \ref{Section 6}, by applying the group-theoretic results proved in Section \ref{Section 5} and also Proposition \ref{prop-two} (ii), we prove Theorem \ref{second generalization}.

\section{Proof of Theorem \ref{first generalization}}\label{Section 2}

\begin{proof}
We first give a proof for $L=\mathbb{Q}$ and then we show how the proof can be adjusted to the case of a number field $L$ of class number one. We let ${\rm M}_{m\times 1}(\mathbb{Z}/n\mathbb{Z})$ be the collection of $m\times 1$ column vectors with entries in $\mathbb{Z}/n\mathbb{Z}$.

For $r\mid n$, a positive divisor $r$ of $n$, the orbit of $\textbf{r}=\left(\begin{array}{cccc} r&0&\hdots&0\end{array}  \right)^T\in {\rm M}_{m\times 1}(\mathbb{Z}/n\mathbb{Z})$ is $\langle {\textbf{r}} \rangle=\{A\textbf{r};~A\in {\rm GL}_m(\mathbb{Z}/n\mathbb{Z})\}$. (By abuse of notation here we used $r$ both as an integer and also as an element of $\mathbb{Z}/n\mathbb{Z}$.) Note that if $A {\bf r}={\bf s},$ where $ \textbf{s}=\left(\begin{array}{cccc} s_1&s_2&\hdots&s_m\end{array}  \right)^T$, then  $(r, n)\mid (s_1, \ldots, s_m, n)$. Also since $A^{-1} {\bf s}={\bf r}$, we have $(s_1, \ldots, s_m, n)\mid (r, n)$. So $A {\bf r}={\bf s}$ implies that $(r, n)=(s_1, \ldots, s_m, n)$.

The above observation shows that  for two distinct positive divisors of $n$ like $r_1$ and $r_2$ the orbits $\langle {\bf r}_1 \rangle$ and $\langle {\bf r}_2 \rangle$ are disjoint. Indeed, if the two orbits intersect, for instance $A{\bf r}_1=B{\bf r}_2={\bf s}$ for some $A, B\in {\rm GL}_m(\mathbb{Z}/n\mathbb{Z})$, then $(r_1, n)=(r_2, n)= (s_1, \ldots, s_m, n)$, and thus $r_1=r_2$.

Next we note that the two elements $A {\bf r}$ and $B {\bf r}$ in $\langle {\bf r} \rangle$ are equal  if and only if $(n/r)\mid a_{i1}-b_{i1}$ for $1\leq i \leq m$. Since the map sending $A\in {\rm GL}_m(\mathbb{Z}/n\mathbb{Z})$ to $A\in {\rm GL}_m(\mathbb{Z}/(n/r)\mathbb{Z})$ is onto, then for $r\neq n$ with $r\mid n$
the cardinality of $\langle {\bf r} \rangle$ is 
$$\Psi(n/r):=\#\left\{ \left( \begin{array}{c} a_{11}\\ \vdots \\a_{m1} \end{array}  \right)\in {\rm M}_{m\times1}(\mathbb{Z}/(n/r)\mathbb{Z})    ;~\left( \begin{array}{ccc} a_{11}&\cdots&a_{1m}\\ \vdots&\ddots&\vdots \\ a_{m1}&\cdots&a_{mm} \end{array}      \right) \in {\rm GL}_m(\mathbb{Z}/(n/r)\mathbb{Z})    \right\}.$$
For $r=n$, we have $\langle {\bf r} \rangle=1$, and so we define $\Psi(1)=1$. Observe that, for a prime $p$, since the $p^m-1$ possibilities for the first column of matrices in ${\rm GL}_m(\mathbb{Z}/p\mathbb{Z})$ lift to $(p^\alpha)^m-(p^{\alpha-1})^m$ possibilities for the first column of matrices in ${\rm GL}_m(\mathbb{Z}/p^\alpha\mathbb{Z})$, we have $\Psi(p^\alpha)= (p^\alpha)^m-(p^{\alpha-1})^m$.

We claim that $\sum_{r\mid n} \Psi(n/r)=n^m$.  Since $\Psi$ is multiplicative, in order to show this, it would suffice to show it for $n=p^\alpha$, a prime power. We have
$$\sum_{r\mid p^\alpha} \Psi({p^\alpha}/r)= \left((p^\alpha)^m-(p^{\alpha-1})^m\right)+\cdots+ (p^m-1)+1=(p^\alpha)^m.$$

Now since $\sum_{r\mid n} \Psi(n/r)=n^m$, we conclude that the sets $\langle {\bf r} \rangle$ as $r$ varies over distinct divisors of $n$ form a partition of $\left( \mathbb{Z}/n\mathbb{Z}\right)^m$, and thus the number of orbits is equal to $d(n)$.

Next,  for a number field $L$ of class number one, we note that for any integral ideal $\mathfrak{r}\mid (n)$ of $\mathcal{O}_L$, we may choose a representative $r$ so that $\mathfrak{r}=(r)$. To process the argument as the case $L=\mathbb{Q}$, it suffices to note that if $r'=ur$ for some unit $u\in \mathcal{O}_L$, there is a matrix $A\in {\rm GL}_m (\mathcal{O}_L/n\mathcal{O}_L)$ whose $(1,1)$-entry  is $u$ such that $A\textbf{r}=\textbf{r}'$, where $\textbf{r}=\left(\begin{array}{cccc} r&0&\hdots&0\end{array}  \right)^T$  and $\textbf{r}'=\left(\begin{array}{cccc} r'&0&\hdots&0\end{array}  \right)^T$. This, in particular, implies that 
$$
\{A\textbf{r};~A\in {\rm GL}_m (\mathcal{O}_L/n\mathcal{O}_L)\}=\{A\textbf{r}';~A\in {\rm GL}_m (\mathcal{O}_L/n\mathcal{O}_L)\}.
$$
\end{proof}

\begin{remark}
For $L=\mathbb{Q}$ and $k=1$, a short proof of Theorem \ref{first generalization} can be obtained by noticing that the group action can be realized as the action of the Galois group of $x^n-1$ on the $n$-th roots of unity. Now the result follows since the roots of the $d$-th cyclotomic polynomial $\Phi_d(x)$  are those roots of unity that have exactly order $d$, the cyclotomic polynomials $\Phi_d(x)$ are irreducible over $\mathbb{Q}$, and   $x^n-1=\prod_{d\mid n} \Phi_d(x)$.
\end{remark}

\section{Proofs of Theorem \ref{main} and Corollary \ref{distribution}}\label{S3}

To prove Theorem \ref{main}, we require  ``Burnside's Lemma'' as stated below. 

\begin{lemma}[Burnside's Lemma]
Let $G$ be a finite group acting on a finite set $X$, and let $\chi(g)$ be the number of fixed points of $g$ on $X$. Then the number of orbits of $G$ in $X$ is equal to
$$ \frac{1}{|G|} \sum_{g\in G} \chi(g).$$
\end{lemma}
\begin{proof}
See \cite[Proposition 1.1]{Se16}.
\end{proof}

Now we are in a position to prove Theorem \ref{main}.

\begin{proof}[Proof of Theorem \ref{main}]
Write $L=F(Y)$. Let $\mathfrak{p}$ denote an unramified prime in $L/F$, and let $\mathfrak{P}$ be a prime above $\mathfrak{p}$. Let $E_Y$ be the family of polynomial equations defining $Y$. 
For any prime $\mathfrak{p}$ (resp., $\mathfrak{P}$) of $F$ (resp., $L$), we let $S_{Y, \mathfrak{p}}$ (resp., $S_{Y, \mathfrak{P}}$) denote the set of solutions of $E_Y (\modd \mathfrak{p})$ (resp. $E_Y (\modd \mathfrak{P})$) in the residue field $\mathcal{O}_F/\mathfrak{p}$ (resp., $\mathcal{O}_L/\mathfrak{P}$).

For any prime $\mathfrak{P}\mid\mathfrak{p}$, we write ${\rm Frob}_{\mathfrak{P}}$ for the generator of $\Gal((\mathcal{O}_L/\mathfrak{P})/(\mathcal{O}_F/\mathfrak{p}))$. Then we have
$$
N_{\mathfrak{p}}(Y)=|S_{Y, \mathfrak{p}}|
=\#\{y\in S_{Y, \mathfrak{P}};~\text{$y$ is fixed by ${\rm Frob}_{\mathfrak{P}}$}\},
$$
where the last quantity is independent of the choice of $\mathfrak{P}$.  

Now let $\sigma_{\mathfrak{P}}$ be the lift of ${\rm Frob}_{\mathfrak{P}}$ to ${\rm Gal}(F(Y)/F)$ and $\sigma_\mathfrak{p}=\{\sigma_{\mathfrak{P}};~\mathfrak{P} \mid \mathfrak{p}\}$ be the Artin symbol at $\mathfrak{p}$.  For each $m$, let $G(m)$ stand for the set of elements in $G={\rm Gal}(F(Y)/F)$ that fixes exactly $m$ points in $Y$. Then for any unramified $\mathfrak{p}$, we have that $N_{\mathfrak{p}}(Y)=m$ if and only if $\sigma_{\mathfrak{p}}\subseteq G(m)$. As one has
$$
\sum_{N(\mathfrak{p})\le x}N_{\mathfrak{p}}^k(Y)=\sum_{m=0}^{|Y|}\sum_{\substack{N(\mathfrak{p})\le x\\\sigma_{\mathfrak{p}}\subseteq G(m)}}m^k=\sum_{m=1}^{|Y|}m^k\sum_{\substack{N(\mathfrak{p})\le x\\\sigma_{\mathfrak{p}}\subseteq G(m)}}1,
$$
the Chebotarev density theorem yields that
\begin{equation}
\label{right}
\lim_{x\rightarrow \infty}\frac{1}{\pi_{F}(x)}\sum_{N(\mathfrak{p})\le x}N_{\mathfrak{p}}^k(Y)=\sum_{m=1}^{|Y|}m^k\frac{|G(m)|}{|G|}.
\end{equation}
We note that $\chi^k(g)$ is the number of points in $Y\times\cdots\times Y$, the $k$ copies of $Y$, fixed by $g$. Thus, we can rewrite the sum on the right of \eqref{right} as
$$
\sum_{m=1}^{|Y|}m^k\frac{|G(m)|}{|G|}=\frac{1}{|G|}\sum_{g\in G}\chi^k(g).
$$
Now we conclude the proof by applying Burnside's Lemma 
that asserts that the above average is the number of orbits of $G$ in the $k$ copies of $Y$.
\end{proof}

\begin{proof}[Proof of Corollary \ref{distribution}]
The proof follows the method of moments as described on pages 59-61 of \cite{E79}. We observe that by Theorem \ref{main} we have $$\alpha_k:=\lim_{n\rightarrow \infty} \int_{-\infty}^{\infty} z^k d H_n(z)= \lim_{n\rightarrow \infty} \frac{1}{\pi_F(n)} \sum_{N(\mathfrak{p})\leq n} N_{\mathfrak{p}}^k (Y)= M_k(G, Y).$$ Note that $$\alpha_k\ll |Y|^k.$$
Thus, for complex $t$-values $|t|<1$, the series 
$$\sum_{k=0}^{\infty} \alpha_k \frac{(it)^k}{k!}$$
converges absolutely. Hence, by \cite[Lemmas 1.43 and 1.44]{E79}, the $\alpha_k$ determine a unique distribution function $H$ that satisfies the conditions given in Corollary \ref{distribution}. 
\end{proof}

\section{Proofs of Propositions \ref{main-prop} and \ref{prop-two}}\label{S2}

\begin{proof}[Proof of Proposition \ref{main-prop}]
(i) As there are only finitely many primes $p$ with $(p,n)>1$, we may assume that $(p,n)=1$.
 In particular, all summations below are over primes $p$ with $(p,n)=1$.

Since $\Bbb{F}^{\times}_p$ is a cyclic group of order $p-1$, we have 
$$
N_{p}(f_{n,1})=(p-1,n).
$$
Thus,
$$
\sum_{p\le x}N_{p}^k(f_{n,1})=\sum_{\substack{p\le x\\d=(p-1,n)}}d^k=\sum_{d\mid n}d^k\sum_{\substack{p\le x\\d=(p-1,n)}}1
=\sum_{d\mid n}d^k\sum_{\substack{p\le x\\d\mid p-1\\(\frac{p-1}{d},\frac{n}{d})=1}}1,
$$
which, by the M\"obius inversion, is
$$
\sum_{d\mid n} d^k \sum_{\substack{p\le x\\d\mid p-1}}\sum_{e\mid(\frac{p-1}{d},\frac{n}{d})}\mu(e)
=\sum_{\substack{d,e\\de\mid n}}d^k \mu(e)\sum_{\substack{p\le x\\de\mid p-1}}1.
$$

Now by the prime number theorem for arithmetic progressions, the last inner sum is asymptotic to 
$$
\frac{1}{\varphi(de)}\pi(x),
$$
as $x\rightarrow \infty$, which completes the proof.

(ii) We may assume that $(p,na)=1$.
 In particular, all summations below (and also in (iii)) are over primes $p$ with $(p,na)=1$.

It is known that $N_{p}(f_{n,a})\neq 0$ if and only if
$$
a^{\frac{p-1}{d}}\equiv 1\enspace(\modd p),
$$
where $d=(p-1,n)$. Moreover, if $N_{p}(f_{n,a})\neq 0$, then $N_{p}(f_{n,a})= (p-1,n)$ (see \cite[Proposition 4.2.1]{IR90}). Thus, we have
$$
\sum_{p\le x}N_{p}^k(f_{n,a})
=\sum_{\substack{p\le x\\d=(p-1,n)\\a^{\frac{p-1}{d}}\equiv 1\enspace(\modd p)}}d^k
=\sum_{d\mid n}d^k\sum_{\substack{p\le x\\d=(p-1,n)\\a^{\frac{p-1}{d}}\equiv 1\enspace(\modd p)}}1
=\sum_{d\mid n}d^k\sum_{\substack{p\le x\\d\mid p-1\\(\frac{p-1}{d},\frac{n}{d})=1\\a^{\frac{p-1}{d}}\equiv 1\enspace(\modd p)}}1.
$$
Again, the M\"obius inversion yields
\begin{equation}
\label{last sum}
\sum_{p\le x}N_{p}^k(f_{n,a})
=\sum_{d\mid n}\sum_{\substack{p\le x\\d\mid p-1\\a^{\frac{p-1}{d}}\equiv 1\enspace(\modd p)}}\sum_{e\mid(\frac{p-1}{d},\frac{n}{d})}\mu(e)
=\sum_{\substack{d,e\\de\mid n}}d^k \mu(e)\sum_{\substack{p\le x\\de\mid p-1\\a^{\frac{p-1}{d}}\equiv 1\enspace(\modd p)}}1.
\end{equation}

Now we analyse the last inner sum in \eqref{last sum}. For $d=1$, the sum is equal to 
$$
\sum_{\substack{p\le x\\de\mid p-1}}1
$$
since the condition $a^{p-1}\equiv 1\enspace(\modd p)$ is always valid by Fermat's little theorem. This contributes
\begin{equation}
\label{onee}
\frac{1}{\varphi(de)}\pi(x),
\end{equation}
as $x\rightarrow \infty$. For $d\ge 2$, on the one hand, $de\mid p-1$ implies that $d\mid p-1$, which together with the condition
$$
a^{\frac{p-1}{d}}\equiv 1\enspace(\modd p)
$$ 
asserts that $p$ splits completely in $\Bbb{Q}(\zeta_{d},a^{1/d})/\Bbb{Q}$.  On the other hand, the condition $de\mid p-1$ tells us that the prime $p~ (\neq 2)$ splits completely in $\Bbb{Q}(\zeta_{de})/\Bbb{Q}$. Thus, for $d\geq 2$, the last inner sum in \eqref{last sum} is
\begin{equation}
\label{bigger or equal to 2}
\#\{p\le x; ~\text{$p$ spilts completely in $\Bbb{Q}(\zeta_{de},a^{1/d})/\Bbb{Q}$}\}\sim\frac{1}{d\varphi(de)}\pi(x),
\end{equation}
as $x\rightarrow \infty$, where the asymptotic behaviour is assured by the Chebotarev density theorem for the Galois extension $\Bbb{Q}(\zeta_{de},a^{1/d})/\Bbb{Q}$, and the fact that under given conditions on $a$,  $[\Bbb{Q}(\zeta_{de},a^{1/d}):\Bbb{Q}]=d\varphi(de)$ (see \cite[Lemma 1]{M05}). Applying \eqref{onee} and \eqref{bigger or equal to 2} in \eqref{last sum} and observing that $d^{k-1}=1$ if $d=1$, we conclude the proof.

(iii) It suffices to note that the sum is, in fact, equal to 
$$
\sum_{\substack{p\le x\\d=(p-1,n)\\a^{\frac{p-1}{d}}\equiv 1\enspace(\modd p)}}d^{k_1}d^{k_2}.
$$
Now the result follows from part (ii).
%
%
%
\end{proof}

\begin{proof}[Proof of Proposition \ref{prop-two}] 

During the proof we assume that $p \geq 5$ is a prime such that $p\nmid \ell N_E$, where $N_E$ is the conductor of $E$. 

(i) Let $E_p(\mathbb{F}_p)$ be the set of $\mathbb{F}_p$-points of $E_p$ (the reduction modulo $p$ of $E$).  Observe that $N_p(E[\ell])=|E_p(\mathbb{F}_p)[\ell]|$, 
where $E_p(\mathbb{F}_p)[\ell]$ is the set of $\ell$-torsion points of $E_p(\mathbb{F}_p)$. Note that since ${E_p} (\mathbb{F}_p)[\ell] \subseteq {E_p}[\ell] \simeq \mathbb{Z}/\ell\mathbb{Z} \times \mathbb{Z}/\ell\mathbb{Z}$,  ${E_p} (\mathbb{F}_p)[\ell]$ has either $1$, $\ell$, or $\ell^2$ elements. Moreover, it is known that
$N_p(E[\ell])= |E_p(\mathbb{F}_p)[\ell]|=\ell^2 $ if and only if $p$ splits completely in the $\ell$-division field $L=\mathbb{Q}(E[\ell])$ of $E$ (see \cite[Lemma 2]{Mu83}).

If $N_p(E[\ell])=\ell$, then for a prime $\mathfrak{P} \mid p$ we can conclude that $\sigma_\mathfrak{P}$ (the lift of ${\rm Frob}_{\mathfrak{P}}$ to ${\rm Gal}( \mathbb{Q}(E[\ell])/\mathbb{Q})$) can have a representation in the form
\begin{equation}
\label{Artin}
\left( \begin{array}{cc}1 & b \\
0 & c \end{array} \right)\in \GL_2(\mathbb{F}_\ell)\backslash\left\{\left( \begin{array}{cc}1 & 0 \\
0 & 1 \end{array} \right)\right\}
\end{equation}
for some  $b\in \Bbb{F}_{\ell}$ and $c\in \Bbb{F}_{\ell}^{\times}$. Thus, $N_p(E[\ell])=\ell$ if and only if the Artin symbol $\sigma_p$ considered as a conjugacy class of 
$\GL_2(\mathbb{F}_\ell)$ has an element of the form \eqref{Artin}. By the Jordan canonical form, a matrix of the form \eqref{Artin} is  conjugate to either
\begin{equation}
\label{Jordan}
\left( \begin{array}{cc}1 & 1 \\
0 & 1 \end{array} \right)\text{ or }
\left( \begin{array}{cc}1 & 0 \\
0 & c \end{array} \right)
\end{equation}
for some  $c\in \Bbb{F}_{\ell}^{\times}\backslash\{1\}$.
Now from the classification of conjugacy classes of $\GL_2(\mathbb{F}_\ell)$ (see \cite[p. 714, Table 12.4]{L93}), it may be computed that the number of elements of such forms in $\GL_2(\mathbb{F}_\ell)$ is $\ell^3-2\ell-1$. 
(Indeed, the ``unipotent" instance in \eqref{Jordan} contributes $\ell^2 - 1$ conjugate elements, and the ``rational not central"  instances in \eqref{Jordan}  contribute $(\ell-2)(\ell^2+\ell)$ elements.)

Let $\pi_E(x;\ell,i)$ for $0\leq i \leq 2$ be defined as
\begin{equation}
\label{def}
\pi_E(x;\ell,i)=\#\{p\le x;~  N_p(E[\ell])=\ell^i \}.
\end{equation}

The above discussion, together with the Chebotarev density theorem and the fact that by our assumption $[\mathbb{Q}(E[\ell]): \mathbb{Q}]=(\ell^2-\ell)(\ell^2-1)$, yields that, as $x\rightarrow \infty$,
$$
\pi_E(x;\ell,1)\sim \frac{\ell^3-2\ell-1}{(\ell^2-\ell)(\ell^2-1)}\pi(x)\text{ and }
\pi_E(x;\ell,2)\sim \frac{1}{(\ell^2-\ell)(\ell^2-1)}\pi(x).
$$
Hence, as $x\rightarrow \infty$,
$$
\pi_E(x;\ell,0) \sim \frac{\ell^4-2\ell^3-\ell^2+3\ell}{(\ell^2-\ell)(\ell^2-1)}\pi(x).
$$

Clearly, it follows from \eqref{def} that
$$
\sum_{p\le x}N_{p}^k(E[\ell])=1^k\cdot \pi_E(x;\ell,0) +\ell^k\cdot \pi_E(x;\ell,1) +\ell^{2k}\cdot \pi_E(x;\ell,2).
$$
Therefore,
\begin{equation}
\label{main2}
\lim_{x\rightarrow \infty}\frac{1}{\pi(x)}\sum_{p\le x}N_{p}^k(E[\ell])
=\frac{\ell^4-2\ell^3-\ell^2+3\ell}{(\ell^2-\ell)(\ell^2-1)}+\ell^k \frac{\ell^3-2\ell-1}{(\ell^2-\ell)(\ell^2-1)}+\ell^{2k}\frac{1}{(\ell^2-\ell)(\ell^2-1)} .
\end{equation}

(ii) We have 
\begin{equation}
\label{zero}
\sum_{p\leq x} N_p^k(E[\ell])= \sum_{\substack{ p\leq x \\ p\textrm{ splits in } K }} N_p^k(E[\ell])+ \sum_{\substack{ p\leq x \\ p\textrm{ is inert or ramifies in } K }}  N_p^k(E[\ell]).
\end{equation}

It is known that if $p$ is inert or ramifies in $K$, then $p$ is supersingular (\cite[p. 182, Theorem 12]{L87}), which implies that (for $p\geq 5$) $|{E_p}( \mathbb{F}_p)|=p+1$ (\cite[p. 145, Exercise 5.10 (b)]{Si86}) and the odd part of ${E_p}(\mathbb{F}_p)$ is cyclic (\cite[Theorem 1]{Mu87}). So, for odd $\ell$, we have $N_p(E[\ell])=(\ell, p+1)$. 
Following the proof of Proposition \ref{main-prop} (i), we conclude that 
\begin{equation}
\label{one}
\lim_{x\rightarrow \infty}\frac{1}{\pi(x)}\sum_{\substack{ p\leq x \\ p\textrm{ is inert or ramifies in } K }}  N_p^k(E[\ell])= \frac{1}{2} M_k(\ell)= \frac{\ell-2}{2(\ell-1)}+\ell^{k} \frac{1}{2(\ell-1)}.
\end{equation}

For $0\leq i \leq 2$, we let 
$$
\pi_E^s(x;\ell,i)=\#\{p\le x;~p \textrm{ splits in } K \textrm{ and } N_p(E[\ell])=\ell^i \}.
$$

 It follows from the definition that 
\begin{equation}
\label{two}
\sum_{\substack{ p\le x \\ p\textrm{ splits in } K }}N_{p}^k(E[\ell])=1^k\cdot \pi_E^s(x;\ell,0) +\ell^k\cdot \pi_E^s(x;\ell,1) +\ell^{2k}\cdot \pi_E^s(x;\ell,2).
\end{equation}

 Recall that $N_p(E[\ell])=\ell^2$ if and only if
 $p$ splits completely in $L=\mathbb{Q}(E[\ell])$ (\cite[Lemma 2]{Mu83}). Now let $p\mathcal{O}_K = (\pi_p \mathcal{O}_K) (\bar{\pi}_p \mathcal{O}_K)$, then $p\mathcal{O}_L$ splits completely in $L$ if and only if $p\mathcal{O}_K$ splits completely in $L$. Also since,  for odd $\ell$, $L=\mathbb{Q}(E[\ell])=K(E[\ell])$ (\cite[Lamma 6]{Mu83}) and $[K(E[\ell]):K]=\ell^2-1$ (according to the assumption), by an application of the Chebotarev density theorem for the extension $K(E[\ell])/K$,  we have
\begin{align*}
\pi_E^s(x;\ell,2)
&=\#\{p\leq x;~p\mathcal{O}_K~\textrm{splits in } K \textrm { and } p\mathcal{O}_L \textrm{ splits in } \mathbb{Q}(E[\ell])\}\\
&=\frac{1}{2} \# \{ \mathfrak{p}\subset \mathcal{O}_K;~N(\mathfrak{p}) \leq x \textrm { and } \mathfrak{p} \textrm{ splits in } {K}(E[\ell])\}+O\left(\frac{x^{1/2}}{\log{x}}\right)\\
&=\frac{\pi_K(x)}{2(\ell^2-1)}(1+o(1))+O\left(\frac{x^{1/2}}{\log{x}}\right).
\end{align*}

The above asymptotic formula together with applications of the Chebotarev density theorem and the fact that $\pi_K(x) \sim \pi(x)$, as $x\rightarrow \infty$,  result in 
\begin{equation}
\label{three}
\pi_E^s(x;\ell,0)\sim \delta_0^s(\ell) \pi(x),~~\pi_E^s(x;\ell,1)\sim \delta_1^s(\ell) \pi(x),~\textrm{ and }~\pi_E^s(x;\ell,2)\sim \frac{1}{2(\ell^2-1)}\pi(x),
\end{equation}
as $x\rightarrow \infty$, where the densities $\delta_0^s(\ell)$ and $\delta_1^s(\ell)$ exist following the discussion at the beginning of (i). Hence, from \eqref{two} with $k=0$, we have
\begin{equation}
\label{equation1}
\delta_0^s(\ell)+\delta_1^s(\ell)+\frac{1}{2(\ell^2-1)}=\frac{1}{2}.
\end{equation}

Also, from \eqref{two} with $k=1$, we have
\begin{equation}
\label{equation2}
\delta_0^s(\ell)+\ell \delta_1^s(\ell)+\frac{\ell^2}{2(\ell^2-1)}= \lim_{x\rightarrow \infty} \frac{1}{\pi(x)}\sum_{\substack{ p\le x \\ p\textrm{ splits in } K }}N_{p}(E[\ell]).
\end{equation}

For a splitting prime $p$, writing $p\mathcal{O}_K=  (\pi_p \mathcal{O}_K) (\bar{\pi}_p \mathcal{O}_K)$ and denoting the reduction (mod $ \pi_p \mathcal{O}_K$) of $E$  by $E_{\pi_p}(\mathcal{O}_{K}/\pi_p \mathcal{O}_{K})$, we have 
$$N_p(E[\ell])=| {E_p}(\mathbb{F}_p)[\ell]|= |{E}_{\pi_p} (\mathcal{O}_{K}/\pi_p \mathcal{O}_{K})[\ell]|= N_{\pi_p\mathcal{O}_K} (E[\ell]).$$
A similar identity holds by replacing $\pi_p$ with $\bar{\pi}_p$. Thus,
$$\sum_{\substack{ p\le x \\ p\textrm{ splits in } K }}N_{p}(E[\ell])=\frac{1}{2} \sum_{\substack{\mathfrak{p}\subset \mathcal{O}_K \\N(\mathfrak{p})\leq x}} N_\mathfrak{p} (E[\ell])+O\left( \frac{x^{1/2}}{\log{x}}   \right).$$
From this and the fact that $\pi(x) \sim \pi_K(x)$, as $x\rightarrow \infty$, we obtain 
$$ \lim_{x\rightarrow \infty} \frac{1}{\pi(x)}\sum_{\substack{ p\le x \\ p\textrm{ splits in } K }}N_{p}(E[\ell])= \lim_{x\rightarrow \infty} \frac{1}{2\pi_K(x)} \sum_{\substack{\mathfrak{p}\subset \mathcal{O}_K \\N(\mathfrak{p})\leq x}} N_\mathfrak{p} (E[\ell]).$$
Now Theorem \ref{main} yields that 
$$\lim_{x\rightarrow \infty} \frac{1}{2\pi_K(x)} \sum_{\substack{\mathfrak{p}\subset \mathcal{O}_K \\N(\mathfrak{p})\leq x}} N_\mathfrak{p} (E[\ell])=\frac{1}{2} M_1({\rm GL}_1(\mathcal{O}_K/\ell \mathcal{O}_K), \mathcal{O}_K/\ell \mathcal{O}_K). $$
We know that $K$ has class number 1 (see \cite[Appendix C, Example 11.3.1]{Si86}).  Therefore, by Theorem \ref{first generalization}, we have
$$M_1({\rm GL}_1(\mathcal{O}_K/\ell \mathcal{O}_K), \mathcal{O}_K/\ell \mathcal{O}_K)=d_K(\ell),$$
where $d_K(\ell)$ is the divisor function for the number field $K$. Applying this value in \eqref{equation2} yields 
\begin{equation}
\label{equation3}
\delta_0^s(\ell)+\ell \delta_1^s(\ell)+\frac{\ell^2}{2(\ell^2-1)}= \frac{1}{2} d_K(\ell).\end{equation}

Solving the system of equations \eqref{equation1} and \eqref{equation3} yields
$$\delta_0^s(\ell)=\frac{\ell^2-(d_K(\ell)-2)\ell-d_K(\ell)}{2(\ell^2-1)}~\textrm{ and }~ \delta_1^s(\ell)=\frac{d_K(\ell)-2}{2(\ell-1)}.$$

Employing these values in \eqref{three} together with \eqref{two}, \eqref{one}, and \eqref{zero} yield the result.
\end{proof}

\section{Proof of Theorem \ref{second-theorem}}\label{Section 5}
\begin{proof}[Proof]
{\it (i)} Let $F=\mathbb{Q}$ and  $Y=\{\zeta_n^i;~i=1, \ldots, n  \}$ be the set of zeros of the polynomial $f_{n,1}(x)=x^n-1$ in $\overline{\mathbb{Q}}$, where $\zeta_n$ denotes a primitive $n$-th root of unity.  Consider the bijection $\psi: X=\mathbb{Z}/n\mathbb{Z} \rightarrow Y$, where $\psi(i)=\zeta_n^i$ and note that $\phi:G=\left(\mathbb{Z}/n\mathbb{Z}\right)^\times \rightarrow {\rm Gal}(F(Y)/F)$ defined by $\phi(d)=\phi_d$, where $\phi_{d}(\zeta_n^j)=\zeta_n^{jd}$, is a group isomorphism. 
Thus, from Theorem \ref{main} and Proposition \ref{main-prop} (i) we have
$$M_k(G, X)= \lim_{x\rightarrow \infty}\frac{1}{\pi(x)}\sum_{p\le x}N_{p}^k(f_{n,1})
=M_k(n).$$

{\it (ii)}  Let $a$ be a square-free positive integer if $n$ is odd, and let $a$ be a square-free positive integer such that $a\nmid n$ if $n$ is even. Let  the number $a^{1/n}$ be a real solution of the equation $x^n-a=0$.
Let $F=\mathbb{Q}$ and $Y=\{(a^{1/n}\zeta_n^i, \zeta_n^j);~1\leq i, j \leq n  \}$ be the set of zeros of the system of polynomials $f_{n,a}(x)=x^n-a$  and $f_{n, 1}(y)=y^n-1$   in $\overline{\mathbb{Q}} \times \overline{\mathbb{Q}}$. Consider the bijection $\psi: X=\left( \{1\} \times \mathbb{Z}/n\mathbb{Z} \right)  \times \left( \{0\} \times \mathbb{Z}/n\mathbb{Z} \right)  \rightarrow Y$, where $\psi(((1, i), (0, j)))=(a^{1/n}\zeta_n^i, \zeta_n^j)$ and note that  $\phi: G \rightarrow  {\rm Gal}(F(Y)/F)$  defined by $\phi\left( { \left( \begin{array}{cc} 1&0\\b&d \end{array} \right)} \right)=\phi_{b, d}$ is an isomorphism, where $\phi_{b, d} ((a^{1/n} \zeta_n^i, \zeta_n^j))=(a^{1/n} \zeta_n ^{b+id}, \zeta_n^{jd})$.

We note that $N_{p}(Y)$ is the number of solutions $(x, y)$ of $x^n\equiv a\enspace(\modd p)$ and $y^n\equiv 1\enspace(\modd p)$, which is equal to $N_{p}(f_{n,a})N_{p}(f_{n,1})$. Thus, from Theorem \ref{main}  we have
$$
M_k(G, X)=\lim_{x\rightarrow \infty}\frac{1}{\pi(x)}\sum_{p\le x}N_{p}^k(Y)=\lim_{x\rightarrow \infty}\frac{1}{\pi(x)}\sum_{p\le x}\left(N_{p}(f_{n,a})N_{p}(f_{n,1})\right)^k,
$$
where the limit on the right can be computed by Proposition \ref{main-prop} (iii).


{\it (iii)} For $\ell \neq 2$, let $E[\ell]$ be the $\ell$-torsion subgroup of the elliptic curve 
 $E_{17a3}$ (with Cremona label $17a3$),  and, for $\ell =2$, let $E[\ell]$ be corresponded to $E_{11a2}$ (with Cremona label $11a2$). Then ${\rm Gal}(\mathbb{Q}(E[\ell])/\mathbb{Q})\simeq {\rm GL}_2(\mathbb{Z}/\ell \mathbb{Z})$ (see \cite{LMFDB} for details). 
 
 For such $E$, let $F=\mathbb{Q}$ and $Y= E[\ell]$.  Consider the bijection $\psi: X=\mathbb{Z}/\ell\mathbb{Z} \times \mathbb{Z}/\ell\mathbb{Z}  \rightarrow E[\ell]$ and note that $G={\rm GL}_2(\mathbb{Z}/\ell \mathbb{Z}) \simeq_\phi {\rm Gal}(F(Y)/F)$. 
Thus, from Theorem \ref{main} and Proposition \ref{prop-two} (i),  we have
$$ M_k(G, X)= \lim_{x\rightarrow \infty}\frac{1}{\pi(x)}\sum_{p\le x}N_{p}^k(E[\ell])
=\frac{\ell^4-2\ell^3-\ell^2+3\ell}{(\ell^2-\ell)(\ell^2-1)}+\ell^k\frac{\ell^3-2\ell-1}{(\ell^2-\ell)(\ell^2-1)}
+\ell^{2k}\frac{1}{(\ell^2-\ell)(\ell^2-1)} .$$

\end{proof}

\section{Proof of Theorem \ref{second generalization}}\label{Section 6}

(i) Since the corresponding action of $\Gal(F(\mathbb{G}_m[n])/F)$ on $\mathbb{G}_m[n]$ is a realization of the canonical action of $G=(\Bbb{Z}/n\Bbb{Z})^\times$ on $X=\Bbb{Z}/n\Bbb{Z}$, the assertion follows from  Theorem \ref{second-theorem} (i) immediately.

(ii) Let $\mathbb{T}$ over $\mathbb{Q}$ be defined by the equation $x^2-my^2=1$, where $m$ is a square-free integer. Then
$$\mathbb{T}[n]=\left\{ \left(\frac{\zeta_n^i+\zeta_n^{-i}}{2}, \frac{\zeta_n^i-\zeta_n^{-i}}{2\sqrt{m}}  \right);~1\leq i\leq n    \right\}$$
is the set of $n$-torsion points of $\mathbb{T}$. By \cite[Lemma 2.1]{CK12},we know that there is a constant $C$ such that for $(n, C)=1$, we have
$\mathbb{Q}(\mathbb{T}[n])=\mathbb{Q}(\zeta_n+\zeta_n^{-1}, (\zeta_n-\zeta_n^{-1})/\sqrt{m})$ and $[\mathbb{Q}(\mathbb{T}[n]):\mathbb{Q}]=\varphi(n)$. Thus, for $1\leq d \leq n$ with $(d, n)=1$, the maps
$$\sigma_d \left(\frac{\zeta_n+\zeta_n^{-1}}{2}, \frac{\zeta_n-\zeta_n^{-1}}{2\sqrt{m}}  \right)=  \left(\frac{\zeta_n^d+\zeta_n^{-d}}{2}, \frac{\zeta_n^d-\zeta_n^{-d}}{2\sqrt{m}}  \right)$$
give the $\mathbb{Q}$-automorphisms of $\mathbb{Q}(\mathbb{T}[n])$, and therefore the action of $\Gal(\mathbb{Q}(\mathbb{T}[n])/\mathbb{Q})$ on $\mathbb{T}[n]$ is a realization of the action of $G=(\Bbb{Z}/n\Bbb{Z})^\times$ on $X=\Bbb{Z}/n\Bbb{Z}$. Now the result follows from Theorem \ref{second-theorem} (i).

(iii)  Let $E$ be a non-CM elliptic curve defined over $F$, and let $n=\prod_{\ell} \ell$ be square-free. By Serre's open image theorem \cite{Se72}, there exists a constant $C$ such that for $(\ell, C)=1$, we have  $\Gal(F(E[\ell])/F)\simeq {\rm GL}_2(\mathbb{Z}/\ell\mathbb{Z})$. 
We note that
$$
\Gal(F(E[n])/F)\simeq \prod_{\ell\mid n} \Gal(F(E[\ell])/F)
$$
acts on $\prod_{\ell\mid n} \left( \mathbb{Z}/\ell\mathbb{Z} \times \mathbb{Z}/\ell\mathbb{Z} \right)^k$ componentwise (i.e., the action is the product of the actions of  \\$\Gal(F(E[\ell])/F)$ on $\left( \mathbb{Z}/\ell\mathbb{Z} \times \mathbb{Z}/\ell\mathbb{Z} \right)^k$). Thus, we have
\begin{equation}
\label{product}
M_k(E/F, n)=\prod_{\ell \mid n} M_k(E/F, \ell). 
\end{equation}
Now applying \eqref{product}  together with Theorem \ref{second-theorem} (iv) completes the proof.

(iv) The proof follows along the same lines as (iii) via employing Deuring's theorem \cite{D41} on the image of $\Gal(K(E[\ell])/K)$ and Proposition \ref{prop-two} (ii).


\subsection*{Acknowledgement}
The authors would like to thank the referee for the valuable comments and suggestions.

\end{document}